\documentclass[notitlepage,leqno,11pt]{article}

\usepackage{amssymb}
\usepackage{amsfonts}
\usepackage{amscd}
\usepackage{amstext}
\usepackage{latexsym,amssymb}
\usepackage{amsmath}
\usepackage{graphicx}
\usepackage{color}
\usepackage{setspace}

\newcommand{\R}{\mathbb{R}}

\newcommand{\Chi}{\raise 1.5pt\hbox{$\chi$}}

\newcommand{\beq}{\begin{equation}}
\newcommand{\eeq}[1]{\label{#1}\end{equation}}
\newcommand{\beqar}{\begin{eqnarray*}}
\newcommand{\eeqar}{\end{eqnarray*}}

\newcommand{\qed}{\penalty 500\hfill$\square$\par\medskip}

\newcommand\interruptenum{\xdef\savecounter{\theenumi}\end{enumerate}}
\newcommand\continueenum{\begin{enumerate}%
\setcounter{enumi}{\savecounter}\item[]}

\newcommand{\proofof}{}

\renewcommand{\text}[1]{\hbox{#1}}
\renewcommand{\eqref}[1]{(\ref{#1})}

\newcommand{\avg}{\raise 0pt\hbox{$-$}\hskip -10.7pt\int}

\newcounter{stepcounter}

\newtheorem{lem}         {Lemma}[section]
\newtheorem{pro}    [lem]{Proposition}
\newtheorem{thm}    [lem]{Theorem}

\newtheorem{df}     [lem]{Definition}
\newtheorem{ex}     {Example}

\date{}

\title{Some results on Ricatti Equations, Floquet Theory and Applications }

\begin{document}
%\pagewiselinenumbers
%\modulolinenumbers[1]
%\linenumbers

%%\small%%%SMALL

\maketitle

%\setstretch{2}

\begin{center}
{\small Anderson L. A. de Araujo, Ab\'ilio Lemos, Alexandre Miranda Alves \\ Universidade Federal de Vi\c{c}osa, CCE, Departamento de Matem\'atica\\ Avenida PH Rolfs, s/n\\ CEP 36570-900, Vi\c{c}osa, MG, Brasil \\ E-mail: \tt anderson.araujo@ufv.br, abiliolemos@ufv.br, amalves@ufv.br}
\end{center}
\begin{center}
and
\end{center}
\begin{center}
{\small Kennedy Martins Pedroso \\ Universidade Federal de Juiz de Fora, Departamento de Matem\'atica\\ Rua Jos\'e Louren\c{c}o Kelmer, s/n,
Bairro S\~ao Pedro \\ CEP: 36036-900, Juiz de Fora, MG, Brasil\\ E-mail: \tt kennedy.pedroso@ufjf.edu.br}
\end{center}

\

\noindent{\sc Abstract}. In this paper, we present two new results to the classical Floquet theory, which provides the Floquet multipliers for two classes of the planar periodic system. One these results provides the Floquet multipliers independently of the solution of system. To demonstrate the application of these analytical results, we consider a cholera epidemic model with phage dynamics and seasonality incorporated.

\

\noindent {\sc AMS Subject Classification 2010}. Floquet theory; stability; epidemic models.
\


\

\newpage

\section{Introduction}
\noindent

The Floquet theory is concerned with the study of the linear stability of differential equations with periodic coefficients, see \cite{Fl}. This theory focuses on the concept of Floquet multipliers and offers a powerful means to analyze nonautonomous, periodic differential equations. However, it is very difficult to determine the Floquet multipliers of general linear periodic systems. Except for a few special cases, which include second-order scalar equations and systems of Hamiltonian type or canonical forms, very little is known about the analysis of Floquet multipliers. 

For a nonlinear periodic system, if it has a nonconstant periodic solution, its stability can be analyzed by linearization about the periodic solution. The variational system then becomes a linear periodic system, and its Floquet multipliers provide useful
information on the stability of the periodic solution.

Consider the planar system
\begin{equation}\label{s1}
\displaystyle{\begin{array}{ccc}
\dot{u} & = & p_{11}(t)u+p_{12}(t)v\\
 \dot{v} & = &  p_{21}(t)u+p_{22}(t)v,\\
\end{array}}
\end{equation}
where the $p_{ij}$ are continuous real valued $T$-periodic functions with $T>0$. By Floquet theory, see \cite{Fl,H}, there are solutions to the system \eqref{s1}, say, $\varphi_1=(u_1,v_1)$ and $\varphi_2=(u_2,v_2)$, and real numbers $\lambda_1$ and $\lambda_2$ (not necessarily distinct) that satisfy
\begin{equation*}
\displaystyle{\begin{array}{ccc}
\varphi_1(t+T) & = & \lambda_1\varphi_1(t)\\
 \varphi_2(t+T) & = & \lambda_2\varphi_2(t).
\end{array}}
 \end{equation*}
The solutions $\varphi_1$ and $\varphi_2$ are called {\it normal solutions} and the numbers $\lambda_1$ and $\lambda_2$ are called {\it Floquet multipliers}. Set

\[ \phi(t)=\left[\begin{array}{cc}
   u_1(t) & u_2(t) \\
  v_1(t) & v_2(t) \\
     \end{array}
   \right] \,\,\mbox{and}\,\,P(t)=\left[\begin{array}{cc}
   p_{11}(t) & p_{12}(t) \\
  p_{21}(t) & p_{22}(t) \\
     \end{array}
   \right],\]
it is a well-known fact that $\lambda_1$ and $\lambda_2$ are the eigenvalues of the matrix $\phi^{-1}(0)\phi(T)$ and that

\begin{equation}\label{e1}
\lambda_1\lambda_2=\exp\int_{0}^{T}\mbox{trace}P(t)dt.
 \end{equation}
We observe that, if $\lambda_1\neq\lambda_2$, then $\varphi_1$ and $\varphi_2$ are linearly independent. In this case, knowledge of the Floquet multipliers and the values of $\varphi_1$ and $\varphi_2$ for $0 < t < T$ gives information for every solution of \eqref{s1} for all $t$. The calculation of the Floquet multipliers is not routine, since one does not generally know even one nontrivial solution of \eqref{s1}.  A procedure for obtaining the Floquet multipliers and the corresponding normal solutions for \eqref{s1} is possible when an associate Riccati equation to the system \eqref{s1} presents a periodic solution. Following Proctor, see \cite{P}, making the change of coordinate $u=z_1+\sigma z_2$ and $v=z_2$ in \eqref{s1}, where $\sigma$ is a solution of the Riccati equation

\begin{equation}\label{e1a}
\dot{x}=c(t)+b(t)x+a(t)x^2
\end{equation}
with $a(t)=-p_{21}(t)$, $b(t)=p_{11}(t)-p_{22}(t)$ and $c(t)=p_{12}(t)$, that is,

\begin{equation}\label{e2}
\dot{x}=p_{12}(t)+[p_{11}(t)-p_{22}(t)]x-p_{21}(t)x^2,
 \end{equation}
the differential equation in $z$ can be integrated, thus providing the following result: 

\begin{pro} \cite[Theorem 3.2]{P}\label{P1}
\begin{enumerate}
\item[(1)] If $\varphi =  (\varphi_1,\varphi_2)$ is a solution of \eqref{s1}, then $\sigma = \frac{\varphi_1}{\varphi_2}$ is a solution of \eqref{e1a} on any interval on which $\varphi$ does not vanish.
\item[(2)] If $\sigma$ is a solution of \eqref{e1a} on an interval $I$ containing the number $k$, then
\begin{equation}\label{e1b}
\displaystyle{\begin{array}{ccc}
 u(t) & = & \sigma(t)\exp{\int_{k}^{t}p_{21}(s)\sigma(s)+p_{22}(s)ds}\\
 v(t) & = & \exp{\int_{k}^{t}p_{21}(s)\sigma(s)+p_{22}(s)ds}\\
\end{array}}
 \end{equation}
\end{enumerate}
is a solution of \eqref{s1}on $I$.
\item[(3)] If $\sigma$ is a solution of \eqref{e1a} with period $nT$ and $f$ is the mean value of $p_{21}\sigma +p_{22}$ over the period $nT$, where $n$ is a positive integer, then $\exp\left(nTf\right)$ is a Floquet multiplier for \eqref{s1} for the period $nT$ and \eqref{e1b} is a normal solution of \eqref{s1} corresponding to this multiplier.
\end{pro}

Thus, for $n=1$, 

\begin{equation*}
\lambda=\exp{\int_{0}^{T}p_{21}(s)\sigma(s)+p_{22}(s)ds}
\end{equation*}
is a Floquet multiplier associated to solution $\varphi$. We observe that the formula given in \eqref{e1} provides us another multiplier.

Now the question is: how to guarantee the existence of a periodic solution to the equation \eqref{e1a}? The next result, proved by Mokhtarzadeh, Pounark and Razani, provides an answer to the question above. Before enunciating it, we define the following function.

$G:[0,T]\times[0,T]\longrightarrow \R$, such that
\begin{equation*}
\displaystyle{G(t,s)=\left\{\begin{array}{rr}
\frac{1}{1-\exp\int_{0}^{T}b(r)dr}\exp\int_{s}^{t}b(r)dr,  & 0\leq s\leq t\leq T\\
 \frac{\exp\int_{0}^{T}b(r)dr}{1-\exp\int_{0}^{T}b(r)dr}\exp\int_{s}^{t}b(r)dr,  & 0\leq t\leq s\leq T.
\end{array}
\right.}
 \end{equation*}

\begin{pro}\label{P2} [Theorem $3.2$, \cite{MPR}]
Let $a(t),b(t)$ and $c(t)$ be continuous $T$-periodic functions with $\int_0^Tb(t)dt\neq0$. Set
\begin{equation}\label{M_N}
M=\displaystyle\sup_{0\leq t,s\leq T}|G(t,s)|\,\,\mbox{and}\,\,N=\displaystyle\sup_{0\leq t\leq T}\left|\int_{0}^{T}G(t,s)c(s)ds\right|
\end{equation}
and suppose
\begin{equation}\label{esta}
\int_{0}^{T}\left|a(\xi)\right|d\xi \leq\displaystyle\frac{1}{4MN}.
\end{equation}
Then, $x'=c(t)+ b(t)x+a(t)x^2$ has at least a $T$-periodic solution.
\end{pro}

In \cite[Theorem 3.2]{MPR}, the authors defined the following Banach space,
\[
X=\{\phi; \phi\,\,\mbox{is a $T$-periodic continuous real functions on }\mathbb{R}\}.
\]
For $\phi\in X$, defined $|\phi|_{\infty}=\sup_{0\leq t\leq T}|\phi(t)|$ and  
\begin{equation}\label{eq7}
\Omega=\{\phi\in X; |\phi-\psi|_{\infty}\leq N\},
\end{equation}
where $\psi: [0,T]\longrightarrow\mathbb{R}$ is defined by
\begin{equation}\label{eq6}
\psi(t)=\int_{0}^{T}G(t,s)c(s)ds.
\end{equation}
It is easy to see that $\Omega$ is closed, bounded and convex subset of $X$. They also defined the operator $S:\Omega\longrightarrow X$ by 
\begin{equation}\label{eq8}
S(\phi)(t)=\int_{0}^{T}G(t,s)\left[a(s)\phi^{2}(s)+c(s)\right] ds.
\end{equation}
By assumptions \eqref{M_N}, \eqref{esta} and using Schauder Fixed Point Theorem, the authors proved the existence of $x\in \Omega$, such that, $S(
x)=x$, that is, for all $t\in[0,T]$
\begin{equation*}
x(t)=\int_{0}^{T}G(t,s)\left[a(s)x^{2}(s)+c(s)\right] ds.
\end{equation*}

As a consequence of the Propositions \ref{P1} and \ref{P2} we state the following result.

\begin{thm}\label{T1}
Suppose 
\begin{enumerate}
\item[(i)] $\int_{0}^{T}p_{11}(t)-p_{22}(t)dt\neq0$;
\item [(ii)] $\int_{0}^{T}\left|p_{21}(t)\right|dt\leq 1/4MN$.
\end{enumerate}
Then, the Floquet multipliers of the system \eqref{s1} are
\[\lambda_1=\exp{\int_{0}^{T}p_{22}(t)+p_{21}(t)\sigma(t)dt}\,\,\mbox{and}\,\,\lambda_2=\exp{\int_{0}^{T}p_{11}(t)-p_{21}(t)\sigma(t)dt},\]
where $\sigma(t)$ is a $T$-periodic solution of equation \eqref{e2}.
\end{thm}

Now, we state one of the main results of this work,  which is related to the existence of a periodic solution to the Ricatti equation, which satisfies $\int_0^Ta(t)x(t)dt=0$, as a consequence of the coincidence degree theory, proposed by R. E. Gaines and J. L. Mawhin \cite{GM}.

\begin{thm}\label{thmA} Consider $0<T$ and let $a(t),b(t)$ and $c(t)$ be continuous $T$-periodic functions with
\begin{equation}\label{h0}
    \int_0^Ta(t)dt=0,
\end{equation}
\begin{equation}\label{h2}
    |a(t)|\leq A, 0<b \leq b(t), \,\, \int_0^Tc(t)dt=0
\end{equation}
where $b,A>0$ are constants. Then, the Ricatti equation
\begin{equation}\label{eqprinc}
x' = c(t) + b(t)x+a(t)x^2
\end{equation}
 has at least one nontrivial $x$ $T$- periodic solution that satisfies
\begin{equation}\label{ort}
    \int_0^Ta(t)x(t)dt=0~\mbox{and}~|x|_{\infty} \leq \frac{b}{2A}.
\end{equation}
\end{thm}

As a consequence of the Theorem \ref{thmA}, we obtain the Floquet multipliers explicitly. In this case, they do not depend on the solution of the Riccati equation.

\begin{thm}\label{T3}
Suppose 
\begin{enumerate}
\item[(i)] $\int_{0}^{T}p_{21}(t)dt=\int_{0}^{T} p_{12}(t)dt=0$ and $|p_{21}(t)|\leq A$, where $A>0$ is a constant;
\item [(ii)] $p_{11}(t)-p_{22}(t)\geq b>0$, where $b$ is a constant.
\end{enumerate}
Then the Floquet multipliers of the system \eqref{s1} are
\[\lambda_1=\exp{\int_{0}^{T}p_{11}(t)dt}\,\,\mbox{and}\,\,\lambda_2=\exp{\int_{0}^{T}p_{22}(t)dt}.\]
\end{thm}

We observe that the Proposition \ref{P2} guarantees the existence of a $T$-periodic solution, but does not explicitly provide such a solution. Therefore, the Floquet multipliers provided by Theorem \ref{T1} cannot be obtained explicitly. On the other hand, the Theorem \ref{T3} provides Floquet multipliers that can be obtained explicitly.

In what follows, starting from section $2$, we present some concepts and results of coincidence degree theory. In section $3$, we prove the theorem \ref{T1}, \ref{thmA} and \ref{T3}. In section $4$, we present some stability results of linear and nonlinear systems in terms of Floquet multipliers. Finally, in section $5$, we apply Theorem \ref{T1} in a mathematical model for cholera.

%%%%%%%%%%%%%%%%%%%%%%%%%%%%%%%%%%%%%%%%%%%%%%%%%%%%%%%%%%%%%%%%

\section{Some concepts and results about coincidence degree}
\noindent 

The method to be used in this paper involves the applications of the continuation theorem of coincidence degree. In order to make this presentation as self-contained as possible, we introduce a few concepts and results about the coincidence degree. For further details, see R. E. Gaines and J. L. Mawhin \cite{GM}.

\begin{df}
Let $X$, $Y$ be real Banach spaces, $L: DomL\subset X \rightarrow Y$ be a linear mapping. The mapping $L$ is said to be a Fredholm mapping of index zero, if
\[\dim KerL = codim ImL <+\infty\]
and $ImL$ is closed in $Y$.
\end{df}

If $L$ is a Fredholm mapping of index zero, then there are continuous projectors $P : X\rightarrow X$ and $Q: Y\rightarrow Y$, such that
\[ImP = Ker L\]
and
\[KerQ= ImL = Im(I - Q).\]
It follows that the restriction $L_P$ of $L$ to $DomL\cap KerP : (I - P)X \rightarrow ImL$ is invertible. Denote the inverse of $L_P$ by $K_P$.

\begin{df}
A continuous mapping $N :X \rightarrow Y$ is said to be $L$-compact on $\overline{\Omega}$, if $\Omega$ is an open bounded subset of $X$, $QN(\overline{\Omega})$ is bounded and $K_P(I - Q)N :\overline{\Omega} \rightarrow X$ is compact.
\end{df}

Since $ImQ$ is isomorphic to $KerL$, there is an isomorphism $J : ImQ \rightarrow KerL$.
We shall be interested in proving the existence of solutions for the operator equation
\begin{equation}\label{1}
    Lx = Nx,
\end{equation}
where a solution is an element of $Dom L \cap \overline{\Omega}$ which verifies (\ref{1}).

The following results is due to R. E. Gaines and J. L. Mawhin \cite{GM}.
\begin{pro}[Mawhin's Continuation Theorem]\label{equiv}
 Let $L$ be a Fredholm mapping of index $0$ and let $N$ be $L$-compact on $\bar{\Omega}$. Suppose that
\begin{enumerate}
    \item[(i)] For each $\lambda \in (0,1)$, $x \in \partial \Omega$
    \[Lx \neq \lambda Nx.\]
     \item[(ii)] $QNx\neq 0$ for each $x \in Ker L \cap \partial \Omega$ and
      \[deg(JQN, \Omega \cap Ker L, 0)\neq 0,\] 
 \hspace{-1cm} where $J : ImQ \rightarrow KerL$ is an isomorphism.
    \end{enumerate}
     Then, the equation $Lx=Nx$ has at least one solution in $Dom L \cap \bar{\Omega}$.
\end{pro}

%%%%%%%%%%%%%%%%%%%%%%%%%%%%%%%%%%%%%%%%%%%%%%%%%%%%%%%%%%%%%%%%%%%
\section{Proof of Theorems \ref{T1}, \ref{thmA} and \ref{T3}}

\subsection{Proof of Theorem \ref{T1}}
\begin{proof}{}
Due to $(i)$ and $(ii)$, the Proposition \ref{P2} guarantees that the equation \eqref{e2} has a $T$-periodic solution $\sigma(t)$.  Thus, by Proposition \ref{P1}, a Floquet multiplier of the system \eqref{s1} is 
\[\lambda_1=\exp{\int_{0}^{T}p_{22}(t)+p_{21}(t)\sigma(t)dt}.\] Now, we use the equation \eqref{e1} to determine the other multiplier, which is
\[\lambda_2=\exp{\int_{0}^{T}p_{11}(t)-p_{21}(t)\sigma(t)dt}.\]
\end{proof}

\subsection{Proof of Theorem \ref{thmA}}

\begin{proof}{}  Consider the following Banach spaces
%\[X = \{x|x \in C^1(\R,\R), x(t + T) = x(t), \textup{for all} \ t \in \R\}\]
%and
\[X = \{x|x \in C(\R,\R), x(t + T) = x(t), \textup{for all} \ t \in \R\} \cap \left\{x|\int_0^Ta(t)x(t)dt=\int_0^Tc(t)x(t)dt=0\right\},\]
and
\[Y = \{x|x \in C(\R,\R), x(t + T) = x(t), \textup{for all} \ t \in \R\},\]
with the norm
%\[\|x\|_X = \max\{|x|_{\infty}, |x'|_{\infty}\}\]
%and
\[\|x\|_X=\|x\|_Y = |x|_{\infty},\]
where $|x|_{\infty} = \displaystyle\max_{t \in [0,T]}|x(t)|$.

Define a linear operator $L: DomL\subset X \rightarrow Y$ by setting
\[Dom L = \{x| x \in X, x' \in C(\R,\R)\}\]
and for $x \in Dom L$,
\begin{equation*}
Lx = x'.
\end{equation*}

We also define a nonlinear operator $N : X \rightarrow Y$ by setting

\begin{equation*}
Nx = c(t) + b(t)x+a(t)x^2.
\end{equation*}

It is not difficult to see that, by (\ref{h0}),
\[Ker L = \R, \ \textup{and} \ Im L = \left\{y| y \in Y, \int_0^Ty(s)ds = 0\right\}. \]
Thus the operator $L$ is a Fredholm operator with index zero.

Define the continuous projector $P: X \rightarrow Ker L$ and the averaging projector $Q: Y \rightarrow Y$ by setting
\[Px(t) = x(0)\]
and
\[Qy(t) = \frac{1}{T}\int_0^Ty(s)ds.\]

Hence, $ImP = Ker L$ and $Ker Q = ImL$. Denoting by $K_P: ImL \rightarrow Dom L \cap Ker P$ the inverse of $L|_{Dom L \cap Ker P}$, we have
\[K_Py(t)= \int_0^ty(s)ds.\]

Then $QN: X \rightarrow Y$ and $K_P(I - Q)N :X \rightarrow X$ read
\[QNx = \frac{1}{T}\int_0^Tb(s)x(s)ds + \frac{1}{T}\int_0^Ta(s)x^2(s)ds,\]
\[K_P(I - Q)Nx(t) =  \int_0^tc(s)ds + \int_0^tb(s)x(s)ds + \int_0^ta(s)x^2(s)ds - tQNx.\]

Clearly, $QN$ and $K_P(I - Q)N$ are continuous. By using Arzela-Ascoli theorem, it is not difficult to prove that $\overline{K_P(I - Q)N(\overline{\Omega})}$ is compact for any open bounded set $\Omega \subset X$. 
Indeed, let $L>0$ such that $|x|_{\infty}\leq L$, for each $x \in \overline{\Omega}$. Notice that $QN(\overline{\Omega})$ is bounded and $|QNx|_{\infty}\leq |b|_{\infty}L+|a|_{\infty}L^2$, for all $x \in \overline{\Omega}$. Hence, $K_P(I - Q)N(\overline{\Omega})$ is uniformly bounded and
\[
|K_P(I - Q)Nx|_{\infty}\leq 2T(|c|_{\infty}+|b|_{\infty}L+|a|_{\infty}L^2), \forall x \in \overline{\Omega}.
\]
Now, let $t_1,t_2 \in [0,T]$, if we suppose $t_2<t_1$, we obtain
\[
\begin{array}{l}
|K_P(I - Q)Nx(t_1)-K_P(I - Q)Nx(t_2)|\\
=\left|\int_{t_2}^{t_1}c(s)ds + \int_{t_2}^{t_1}b(s)x(s)ds + \int_{t_2}^{t_1}a(s)x^2(s)ds - (t_1-t_2)QNx\right|\\
\leq 2(|c|_{\infty}+|b|_{\infty}L+|a|_{\infty}L^2)|t_1-t_2|, \forall x \in \overline{\Omega}.
\end{array}
\] 
Therefore, $K_P(I - Q)N(\overline{\Omega})$ is a equicontinuous set of $C([0,T])$(hence of $X$). By using Arzela-Ascoli theorem, $\overline{K_P(I - Q)N(\overline{\Omega})}$ is compact.
Therefore, $N$ is $L$-compact on $\overline{\Omega}$ with any open bounded set $\Omega \subset X$.

As $\frac{b}{2A}>0$, we consider 
\begin{equation}\label{i1}
    \Omega_{d}:=\{x \in X| |x|_{\infty} < \frac{b}{2A}\},
\end{equation}
that is an open set in $X$.

Notice that
\begin{equation}\label{p1}
    b(t) - a(t)\frac{b}{2A}>0.
\end{equation}
    Indeed,
\[b(t) - a(t)\frac{b}{2A} \geq  b-A.\frac{b}{2A}= b-\frac{b}{2} = \frac{b}{2}>0.\]

Notice that
\begin{equation}\label{p2}
    b(t) + a(t)\frac{b}{2A}>0.
\end{equation}
Indeed,
\[b(t) + a(t)\frac{b}{2A} \geq b -A\frac{b}{2A}=b -\frac{b}{2}=\frac{b}{2}>0.\]

Let $0 < \lambda <1$ and $x$ such that
\[x'=\lambda c(t) + \lambda b(t)x+\lambda a(t)x^2.\]

By multiplying by $x$ and integrand of $0$ to $T$, we have
\[0=\int_0^Tx'xdt = \lambda \int_0^Tc(t)xdt + \lambda \int_0^Tb(t)x^2 + a(t)x^3dt.\]
That is,
\[0=\int_0^Tc(t)xdt + \int_0^Tx^2(b(t) + a(t)x)dt=\int_0^Tx^2(b(t) + a(t)x)dt.\]

By (\ref{i1}), if $x \in \partial \Omega_{d}$, we have $|x|_{\infty}= \frac{b}{2A}$, and we obtain 
\[0\geq \int_0^Tx^2(b(t) - |a(t)||x|_{\infty})dt. \]
By (\ref{h2}), we have 
\[0\geq \int_0^Tx^2(b - A\frac{b}{2A})dt\]
\[\geq \int_0^Tx^2\frac{b}{2}dt>0.\]
But this is a contradiction. Therefore, the condition ($1$) of Proposition \ref{equiv} holds for $\Omega_{c}$. 
Take $x \in \partial \Omega_{d} \cap Ker L$. Thus, we have $x = -\frac{b}{2A}$ or $x= \frac{b}{2A}$.

If $x= -\frac{b}{2A}$, by (\ref{p1}), we have
\[b(t) - a(t)\frac{b}{2A}>0.\]
Hence,
\begin{equation}\label{Q-}
QNx = \frac{1}{T}\int_0^T-\frac{b}{2A}\left(b(t) - a(t)\frac{b}{2A}\right)dt < 0.
\end{equation}

  If $x= \frac{b}{2A}$, by (\ref{p2}), we have
\[b(t) + a(t)\frac{b}{2A}>0.\]
Hence,
\begin{equation}\label{Q+}
QNx = \frac{1}{T}\int_0^T\frac{b}{2A}\left(b(t) + a(t)\frac{b}{2A}\right)dt > 0.
\end{equation}

Then, for each $x \in \partial \Omega_{d} \cap Ker L$, we have
\begin{equation}\label{i3}
QNx = \frac{1}{T}\int_0^Tx\left(b(t) + a(t)x\right)dt \neq 0.
\end{equation}
Therefore, the condition ($2$) of Proposition \ref{equiv} holds for $\Omega_{d}$.

Define a continuous function $H(x,\mu)$ by setting
\[H(x,\mu) = (1-\mu)x + \mu \frac{1}{T}\int_0^Tx\left(b(t) + a(t)x\right)dt, \,\,\,\, \mu \in [0,1].\]
It follows from \eqref{Q-}, \eqref{Q+} and (\ref{i3}) that
\[H(x,\mu)\neq 0, \ \textup{for all} \ x \in \partial \Omega_{d} \cap Ker L.\]
Hence, using the homotopy invariance theorem, we have
\[deg(QN, \Omega_{d} \cap Ker L, 0) = deg\left(\frac{1}{T}\int_0^Tx\left(b(t) + a(t)x\right)dt, \Omega_{d} \cap Ker L, 0\right)\]
\[=deg\left(x, \Omega_{d} \cap Ker L, 0\right) =-1 \neq 0.\]

In view of all the discussions above, we conclude from Proposition \ref{equiv} that the equation (\ref{eqprinc}) has a solution in $Dom L \cap\bar{\Omega}_{d}$.
\end{proof}

\subsection{Proof of Theorem \ref{T3}}
\begin{proof}{}
Due to $(i)$ and $(ii)$, the Theorem \ref{thmA} guarantees that the equation \eqref{e2} has a $T$-periodic solution $\sigma(t)$ and $\int_{0}^{T}p_{21}(t)\sigma(t)dt=0$. Thus, by Proposition \ref{P1}, a Floquet multiplier of the system \eqref{s1} is 
\[\lambda_1=\exp{\int_{0}^{T}p_{11}(t)dt}.\] Now, we use the equation \eqref{e1} to determine the other multiplier, which is
\[\lambda_2=\exp{\int_{0}^{T}p_{22}(t)dt}.\]
\end{proof}

%%%%%%%%%%%%%%%%%%%%%%%%%%%%%%%%%%%%%%%%%%%%%%%
%%%%%%%%%%%%%%%%%%%%%%%%%%%%%%%%%%%%%%%%%%%%%
\section{Stability of linear and nonlinear systems}
$\indent$

The following theorem provides details about the stability of the system \eqref{s1} in terms of Floquet exponents. We refer to Theorem 7.2 on page 120 of Hale's book \cite{Ha}.

\begin{thm}\label{T4}
\begin{enumerate}
\item[(i)]  A necessary and sufficient condition that the system \eqref{s1} is uniformly stable is that the Floquet multipliers of the system \eqref{s1} have modulii $\leq 1$ and the ones with modulii $=1$ have multiplicity $1$. 
\item[(ii)]  A necessary and sufficient condition that the system
\eqref{s1} is uniformly asymptotically stable is that all Floquet multipliers of the system \eqref{s1} have modulii $<1$.
\end{enumerate}
\end{thm}

Now, consider the system
\begin{equation}\label{s0}
	\dot{X}=A(t)X+F(t,X),
\end{equation}
where $A(t)$ is an $n\times n$ continuous matrix function, and $F(t,X)$ is continuous in $t$ and $X$ 
and Lipschitz-continuous in $X$ for all $t\in \mathbb{R}$ and $X$ in a neighbourhood of $X=0$. Moreover, we assume that
\begin{equation}\label{s0.2}
	\lim\limits_{|X|\to 0}\frac{|F(t,X)|}{|X|}=0 \mbox{ uniformly in } t.
\end{equation}
Notice that the condition in \eqref{s0.2} implies that $X=0$ is a solution to system \eqref{s0}. Then, we
have the following theorem on the behaviour of the trivial solution $X=0$. This result is an
extended version of \cite[Theorem 2.4]{Ha} or \cite[Theorem 7.2]{V}.
\begin{thm}\label{teo.stab}
If the trivial solution $X=0$ of the system $\dot{X}=A(t)X$ is uniformly asymptotically
stable for $t \geq 0$, then the trivial solution of \eqref{s0} is also uniformly asymptotically
stable. If the trivial solution $X=0$ of the system $\dot{X}=A(t)X$ is unstable, then the trivial
solution of \eqref{s0} is also unstable.
\end{thm}

In the following example, we study the stability of a planar system, whose Floquet multipliers are obtained by Theorem \ref{T3}.

\begin{ex}\label{ex1}
Consider the planar system
\begin{equation}\label{s2}
\displaystyle{\begin{array}{ccc}
\dot{u} & = & (m(t)-A)u+\sin(t)v\\
 \dot{v} & = &  a(t)u+(m(t)-B)v,\\
\end{array}}
 \end{equation} 
where $m(t)$ is a continuous function and $2\pi$-periodic,

\[a(t)=\dfrac{(\alpha+1)\sin(t)}{\beta+\cos(t)},\]
$\beta >1$ and $A-B>0$. The Riccati equation associated with this system is
\begin{equation*}
y'=-a(t)y^2+(A-B)y+\sin(t).
\end{equation*}
Since $\int_{0}^{2\pi}a(t)dt=\int_{0}^{2\pi}\sin(t)dt=0$, $b(t)=A-B>0$ and $a(t)$ is limited, according to Theorem \ref{thmA}, that there is a $2\pi$-periodic solution $\sigma(t)$ with $\int_{0}^{2\pi}a(t)\sigma(t)dt=0$. Now, by Theorem \ref{T3}, the Floquet multipliers are
\begin{equation*}
\displaystyle{\begin{array}{ccc}
\lambda_1 & = & \displaystyle\exp\int_0^{2\pi} p_{11}(t)dt\\
\lambda_2 & = & \displaystyle\exp\int_0^{2\pi} p_{22}(t)dt,\\
\end{array}}
\end{equation*}
that is,
\begin{equation*}
\displaystyle{\begin{array}{ccc}
\lambda_1 & = & \displaystyle\exp\int_0^{2\pi} (m(t)-A)dt\\
\lambda_2 & = & \displaystyle\exp\int_0^{2\pi} (m(t)-B)dt.\\
\end{array}}
\end{equation*}
Therefore,
\begin{equation*}
\lambda_1= e^{-2\pi A}e^{\int_0^{2\pi}m(t)dt}\quad\mbox{and}\quad \lambda_2=e^{-2\pi B} e^{\int_0^{2\pi} m(t)dt}.
\end{equation*}
Observe that, if $\int_0^{2\pi}m(t)dt<2\pi A$ and  $\int_0^{2\pi}m(t)dt<2\pi B$, then the system \eqref{s2}, according to Theorem \ref{T4},  is uniformly asymptotically stable.
\end{ex}

%\begin{ex}\label{ex2}
%Consider the planar system
%\begin{equation}\label{s30}
%\displaystyle{\begin{array}{ccc}
%\dot{u} & = & m(t)u+c(t)v\\
 %\dot{v} & = &  a(t)u+n(t)v,\\
%\end{array}}
 %\end{equation}
%where $m$, $n$ and $c$ are continuous function and $T$-periodic, and 
%$$a(t)=\dfrac{\alpha\big(\beta +g(t)+(\alpha+1)c(t)\big)}{(\beta+g(t))^2},$$
%such that $-\beta<\displaystyle\min_{0\leq t\leq T}g(t)$ and $g(t)=-\int_0^tc(s)ds$. The Riccati equation associated with this system is
%\begin{equation*}
%y'=-a(t)y^2+b(t)y+c(t),
%\end{equation*}
%where $b(t)=m(t)-n(t)$, whose $T$-periodic solution is
%$$\sigma(t)=\frac{1}{\alpha}(\beta+g(t)).$$
%By Proposition \ref{P1} item $(3)$ and equation \ref{e1}, the Floquet multipliers are
%\begin{equation*}
%\displaystyle{\begin{array}{ccc}
%\lambda_1 & = & \displaystyle\exp\int_0^{T} m(t)-a(t)\sigma(t)dt\\
%\lambda_2 & = & \displaystyle\exp\int_0^{T} n(t)+a(t)\sigma(t)dt.\\
%\end{array}}
%\end{equation*}
%By slightly simple calculation we have $\int_{0}^{T}a(t)\sigma(t)dt=T$. Therefore,
%
%\begin{equation*}
%\lambda_1= e^{-T}\exp{\int_0^{T} m(t)dt}\quad\mbox{and}\quad \lambda_2= e^{T}\exp{\int_0^{T} n(t)dt}.
%\end{equation*}
%The study of stability of system \eqref{s30} is obtained by Theorem \ref{T4} according to the behavior of the functions $m$ and $n$.
%\end{ex}

%%%%%%%%%%%%%%%%%%%%%%%%%%%%%%%%%%%%%%%%%%%%%%%
%%%%%%%%%%%%%%%%%%%%%%%%%%%%%%%%%%%%%%%%%%%%%
\section{An application to cholera modeling}
$\indent$

In this section, we propose the analysis of a mathematical model for cholera dynamics with seasonal oscillation, studied by \cite{TW}. Cholera is a severe intestinal infection caused by the bacterium \textit{Vibrio cholerae}. Many epidemic models have been published, but Code\c{c}o \cite{Cod} was the first to explicitly incorporate bacterial dynamics into a SIR epidemiological model.  

The following new model is a significant extension of Code\c{c}o's model \cite{Cod} that incorporates the phage dynamics and the seasonal oscillation of cholera transmission, as proposed in \cite{TW}:
\begin{equation}
\label{OriginalBIC}
\left\{
\begin{array}{lcl}
\frac{dS}{dt}&=&n(H-S) - d\frac{B}{K+B}S,\\
\frac{dI}{dt}&=&d \frac{B}{K+B}S - rI,\\
\frac{dB}{dt}&=&eI-mB- \delta\frac{B}{\tilde{K}+B}P,\\
\frac{dP}{dt}&=&\xi\,I +\kappa\frac{B}{\tilde{K}+B}P-\nu\,P,
 \end{array}
\right.
\end{equation}  
where
\begin{itemize}
	\item $S$ is the susceptible human population,
	\item $I$ is the infectious human population,
	\item $B$ and $P$ are the concentrations of the pathogen (i.e. vibrio) and the phage, respectively in the contaminated water,
	\item the total human population, $H$, is assumed to be a constant,
	\item $n$ denotes the natural human birth/death rate,
	\item $d$ denotes the human contact rate to the vibrio,
	\item $\delta$ is the death rate of the bacteria due to phage predation,
	\item $\kappa$ is the growth rate of the phage due to feeding on the vibrio,
	\item $e$ and $\xi$ are the rates of human contribution (e.g. by shedding) to the pathogen and the phage, respectively,
	\item $m$ and $\nu$ are the natural death
rates of the vibrio and the phage, respectively. 
	\end{itemize}
	
In addition, $r=n+\gamma$ with $\gamma$ as the recovery rate, and $K$ and $\tilde{K}$ as the half saturation rates of the vibrio in the interaction with human and phage, respectively.

In \cite{TW}, the authors investigated the impact of seasonality on cholera dynamics, particularity when they examined the periodic variation of three parameters, $m$, $e$ and $d$, and applied the results from the Floquet theory in the analysis, as follows. Firstly, the authors consider that the parameter $m$ is a positive periodic function of time, $m(t)$, which represents a seasonal variation of the extinction rate of the vibrio. On the second scenario, the parameter $e$ is set as a positive periodic function, $e(t)$, which represents a seasonal oscillation of the per capita contamination rate, i.e., the unit rate of human contribution (e.g. shedding) to the pathogen in the environment. On the third scenario, the parameter $d$ is set as a positive periodic function, $d(t)$, which represents a seasonal variation of the contact rate. For an accurate analysis of the dynamics in each previous case, see \cite{TW}.  

In this paper, we consider the following scenario, by setting the parameters $m$, $e$ and $d$ as positive periodic functions $m(t)$, $e(t)$ and $d(t)$ simultaneously.

It is clear that $E_0 = (H, 0, 0, 0)$ is the unique disease free equilibrium (DFE) of the system. For ease of discussion, we translated the DFE to the origin via a change of variable
by $\overline{S} = H − S$. Then, with linearization at $(0, 0, 0, 0)$, the original system becomes
\begin{equation}\label{sist2}
\left\{
\begin{array}{lcl}
\frac{d\overline{S}}{dt}&=&-n\overline{S} + \frac{d(t)H}{K}B+\left(\frac{d(t)B(H-\overline{S})}{K+B} -\frac{d(t)H}{K}B\right),\\
\frac{dI}{dt}&=&-rI + \frac{d(t)H}{K}B +\left(\frac{d(t)B(H-\overline{S})}{K+B} -\frac{d(t)H}{K}B\right),\\
\frac{dB}{dt}&=&e(t)I-m(t)B- \delta\frac{B}{\tilde{K}+B}P,\\
\frac{dP}{dt}&=&\xi\,I +\kappa\frac{B}{\tilde{K}+B}P-\nu\,P.
 \end{array}
\right.
\end{equation}  
Thus, system \eqref{sist2} can be written in a compact form

\begin{equation*}
	\dot{X}=A(t)X+F(t,X),
\end{equation*}
with $X=(\bar{S},I,B,P)^{T}$ and the matrix
\begin{equation}\label{s4}
A(t)=\left(
\begin{array}{cccc}
-n & 0 & \frac{d(t)H}{K} & 0\\
0 & -r & \frac{d(t)H}{K} & 0\\
0 & e(t) & -m(t) & 0\\
0 & \xi & 0 & -\nu
 \end{array}
\right).
\end{equation}  
By \cite{TW}, it is straightforward to check that $\lim\limits_{|X|\to 0}\frac{|F(t,X)|}{|X|}=0$ uniformly in $t$. Based on Theorem \ref{teo.stab}, we only need to consider the periodic linear system
\begin{equation}\label{s3.0}
	\dot{X}=A(t)X,
\end{equation}
 where the matrix $A(t)$ is given by \eqref{s4}.
	
Notice that the matrix $A(t)$ has a block tridiagonal structure. From \cite[Theorem 2.5 and Corollary 2.5]{TW}, two Floquet exponents of the system \eqref{s3.0} are given by $−n$ and $−\nu$; the other two Floquet exponents are determined by the matrix block $\left(
\begin{array}{cc}
 -r & \frac{d(t)H}{K}\\
e(t) & -m(t)\\
 \end{array}
\right).$ Hence, its stability depends on the $2\times2$ sub-system
\begin{equation}\label{s3.1}
	\dot{Y}=\left(
\begin{array}{cc}
 -r & \frac{d(t)H}{K}\\
e(t) & -m(t)\\
 \end{array}
\right)Y.
\end{equation}
We consider $m(t)$, $e(t)$ and $d(t)$, such that 
\begin{equation}\label{s3.2}
A=\max\limits_{0\leq t\leq T}d(t),
\end{equation}
\begin{equation}\label{s3.2.2}
E=\max\limits_{0\leq t\leq T}e(t)
\end{equation}
 and
\begin{equation}\label{s3.2.3}
m_1=\min\limits_{0\leq t\leq T}m(t),
\end{equation}
satisfying
\begin{equation}\label{s3.2.4}
2EMAT\frac{H}{K} <\min\{m_1,r\}
\end{equation}
and
\begin{equation}\label{s3.2.5}
\int_{0}^{T}e(t)dt\leq \frac{1}{4MN},
\end{equation}
where $M$ and $N$ are defined in \eqref{M_N}.

We have the following lemma to describe the stability of the sub-system \eqref{s3.1}.
\begin{lem}\label{lem}
Suppose \eqref{s3.2.4}, \eqref{s3.2.5} and $r\neq\frac{1}{T}\int_{0}^{T}m(t)dt$. Then, the trivial solution of \eqref{s3.1} is asymptotically stable. 
\end{lem}
\begin{proof}{} We observe that $m(t)$, $e(t)$ and $a(t)$ satisfies the assumptions of Theorem \ref{T1}, with $\int_{0}^{T}(p_{11}-p_{22})dt=-Tr+\int_{0}^{T}m(t)dt\neq 0$ and $\int_{0}^{T}p_{21}(t)dt=\int_{0}^{T}e(t)dt\leq \frac{1}{4MN}$. Then, it follows from Theorem \ref{T1} that the Floquet multipliers of the system \eqref{s3.1} are
\[
\lambda_1=\exp{\left(\int_{0}^{T}(-r-e(t)\sigma(t))dt\right)}
\]
and
\[
\lambda_2=\exp{\left(\int_{0}^{T}(-m(t)+e(t)\sigma(t))dt\right)},
\]
where $\sigma(t)$ is a $T$-periodic solution of equation \eqref{e2}. Since $\sigma \in \Omega$ ($\Omega$ defined in \eqref{eq7}), we obtain
\[
\left|\sigma(t) - \int_{0}^{T}G(t,s)a(s)\frac{H}{K}ds\right|<N.
\] 
Therefore,
\[
|\sigma(t)|\leq  \int_{0}^{T}|G(t,s)||a(s)\frac{H}{K}|ds + N \leq 2MAT\frac{H}{K}.
\] 
Hence, by \eqref{s3.2.4}, we obtain
\[
-r-e(t)\sigma(t)\leq -r+2EMAT\frac{H}{K}<0
\]
and
\[
-m(t)+e(t)\sigma(t)\leq -m_1+2EMAT\frac{H}{K}<0.
\]
Notice that all Floquet multipliers of the system \eqref{s3.1} have modulii $<1$. By Theorem \ref{T4}-($ii$), the trivial solution $X=0$ of the system \eqref{s3.1} is uniformly asymptotically stable for $t\geq 0$. 
\end{proof}

We can summarize our analysis above in the theorem below.
\begin{thm}\label{thm.est}
Assume that system \eqref{OriginalBIC} satisfies \eqref{s3.2.4}, \eqref{s3.2.5} and $r\neq\frac{1}{T}\int_{0}^{T}m(t)dt$. Then, the disease free equilibrium $E_0=(H,0,0,0)$ is uniformly asymptotically stable. 
\end{thm}

%\noindent{\sc Acknowledgement}. 

\end{document}